\documentclass[12pt]{amsart}
\usepackage{amsthm, amssymb, verbatim}  
\usepackage{ulem} 
\begin{document}
\title{Totally Real Mappings and Independent Mappings}
\author{Howard Jacobowitz and Peter Landweber}
\date{\today}
\maketitle

\def\rank{\operatorname{rank}}
\def\cp2{{\mathbf{C}\mathrm{P}^2}}
\def\rp2{{\mathbf{R}\mathrm{P}^2}}
\def\bZ{{\mathbf{Z}}}
\def\codim{\operatorname{codim}}
\def\&{&=& }
\def\ph{pseudo-hermitian\ }
\def\bar{\overline}
\def\spc{strictly pseudo-convex\ }
\def\psh{plurisubharmonic}
\def\spsh{strictly plurisubharmonic}
\def\nbd{neighborhood}
\def\iff{if and only if }
\def\ctm{\mathbf{C}\otimes TM}
\def\comtanX{C\otimes T(X)}
\def\comtanO{C\otimes T(\Omega)}
\def\dpartial{\bar{\partial}}
\def\bR{\mathbf{R}}
\def\bC{\mathbf{C}}
\def\tr{totally real }
\def\im{independent map\ }
\def\ims{independent maps}
\def\bRP{\bf R \mathrm{P}}
\def\varep{\varepsilon}
\def\ctmnp{{\bf C}\otimes TM}  

\newtheorem{theorem}{Theorem}[section]
\newtheorem*{theoremstar}{Theorem}
\newtheorem{lemma}{Lemma}[section]
\newtheorem{definition}{Definition}
\newtheorem{corollary}{Corollary}[section]
\newtheorem*{remark}{Remark}  


\section {Introduction}
We consider two classes of smooth maps $M^n\to \bC ^N$.  All manifolds are assumed to be connected (unless otherwise mentioned) and to have countable topology.

\begin{definition}
A map $M^n\to \bC ^N$ is called a \textbf{\tr immersion (embedding)} if $f$ is an immersion (embedding) and for $f_*: TM\to T\bC ^N$ we have 
\begin{equation} \label{tri-equation}
f_*(TM)\cap Jf_*(TM)=\{ 0 \}.
\end{equation}
Here we have identified $\bC ^N$ with $\bR^{2N}$ together with the natural anti-involution $J$.
\end{definition}

\begin{definition}
A map 
$f:M^n\to\bC^N$ is called an \textbf{\im} 
if
\[
df_1(p)\wedge\cdots \wedge df_N (p)\neq 0
\]
for $f=(f_1,\ldots ,f_N)$ and for all $p\in M$.
\end{definition}

We are interested in the optimal value of $N$ for all manifolds of dimension $n$.
Sections \ref{existence} and \ref{optimality}
provide an exposition of \cite{HJL}, which gives some details not presented here. 
Section \ref{new}
discusses a special case where our two types of maps are related in a perhaps unexpected way.


\section{Existence}\label{existence}

\begin{theorem}\label{tri}
Any map $f:M^n\to \bC ^N$ may be approximated by a
\tr embedding, provided $N\geq [\frac {3n} 2]$ and $n\geq 2$.
\end{theorem}

\begin{remark}\normalfont
{For $N$ and $n$ satisfying these inequalities, any map of $M$ into $\bC ^N$ may be approximated by an embedding \cite{Whitney}
and in particular any \tr immersion  may be approximated by a \tr embedding.}
\end{remark}

\begin{theorem}\label{ind}
Any map $f:M^n\to \bC ^N$ may be approximated by an independent map,
provided $N\leq [\frac {n+1} 2]$.
\end{theorem}

The proofs of these theorems depend on a well-known result from differential topology.  Let $J^1(M,W)$ be the space of one-jets of maps from $M$ to $W$.  Denote the lift of any map
\[
f:  M\to W
\]
by 
\[
j^1(f):  M\to J^1(M,W).
\]
\begin{theoremstar}  If $\Sigma \subset  J^1(M,W)$ is  stratified by locally closed submanifolds   and 
$\dim M<\codim \Sigma$ then there exists some $F:M\to W$ with 
$(j^1(F)M)\cap \Sigma = \varnothing$.
\end{theoremstar}

The proof is straightforward, see for instance \cite{EM}.
  
\bigskip
To prove Theorem \ref{tri} we take $\Sigma$ to be given in local coordinates over an open set $U$ by
\[
\Sigma =\{(p,q,a^1,\ldots ,a^n):p\in U, q\in \bR ^{2N} , \rank \,(A,JA)<2n\} 
\]
where  
\[
A=(a^1\cdots a^n)
\]
is a real $N\times n$ matrix and $(A,JA)$ is the $N\times 2n$ matrix obtained by 
juxtaposition.  
To prove Theorem \ref{ind} we set $r=[\frac {n+1} 2]$ and use
\[
\Sigma =\{(p,q,\alpha ^1,\ldots , \alpha ^n),\,\rank A <r\} 
\]
where $A$ is the complex $r\times n$ matrix
\[
A=(\alpha ^1\cdots \alpha ^n).
\]  
It is easy to verify that these are stratified subsets and that the given values of 
$N$ lead to 
$$\dim M < \codim \Sigma .$$


\section {Optimality}\label{optimality}
To explain our examples, we find necessary bundle-theoretic conditions for \tr immersions and for \ims .

\begin{lemma} \label{QQ}
\begin{enumerate}\item[(a)]
If $M$ has a totally real immersion into $\bC^N$ then there exists a bundle $Q$ of rank $r=N-n$ such that 
\[
(\ctm) \oplus Q \cong N\varepsilon .
\]

\item [(b)]\label{BB}
If $M$ has an independent map into $\bC^N$ then there exists a bundle $B$ of rank $r=n-N$ such that
\[
\ctm \cong N\varepsilon \oplus B.
\]
\end{enumerate}
\end{lemma}

Here $N\varepsilon$ is the trivial complex vector bundle over $M$ of rank $N$.

\begin{remark}\normalfont
{An application of the Gromov h-principle  shows that these conditions are also sufficient.  See a discussion of this in \cite{Korea}.  We will make use of the sufficiency below.}
\end{remark}
\begin{proof}[Proof of Lemma \ref{QQ}]
(a) The condition \eqref{tri-equation} is equivalent to the fiber injectivity of
\[
\phi _f:\ctm\to T^{1,0}(\bC ^N)
\]
where $\phi _f(v)$ is defined, for $v\in \ctm$, by
\[
\phi _f(v)=f_*(v)-iJf_*(v).
\]
Thus if $M$ has a \tr immersion \ into $\bC ^N$ then
\[
(\ctm) \oplus Q \cong N\varepsilon
\]
where $Q$ is the bundle in $T^{1,0}$ normal to $\phi _f(\ctm )$.

\medskip
(b) The map 
\[
\psi _f :\ctm\to T^{1,0}
\]
given by
\[
\psi _f(v)=\sum df_j(v)\partial _{z_j}
\]
is surjective on the fibers.  So
\[
\ctm \cong N\varepsilon \oplus B.
\]
with $B=\ker {\psi _f}$. 
\end{proof}

\subsection{Totally real immersions}

We need to find a manifold of dimension $n$ that does not have a \tr immersion into 
$\bC ^N$ for $N=[\frac {3n} 2]-1$.  We provide four families of examples according to the residue of the dimension of $M$ modulo $4$.  Let
\[
M^{4k} = \cp2 \times \cdots \times \cp2 = (\cp2)^{\times k}
\]
be the product of $k$ copies of the complex projective plane.  The manifolds we use and the ensuing arguments are similar to those given by Forster \cite{F}, but we use orientable manifolds as far as possible.

\begin{theorem}\ \label{M}
\begin{itemize}
	\item 
	$M^{4k} $ does not admit a \tr immersion into $\bC ^N$ for $N=6k-1$.
	
	\item 
	$M^{4k+1}=M^{4k}\times S^1 $ does not admit a \tr immersion into 
	$\bC ^N$ for $N=6k$.

	\item 
	$M^{4k+2}=M^{4k}\times \rp2 $ does not admit a \tr immersion into 
	$\bC ^N$ for $N=6k+2$.
	
	\item 
	$M^{4k+3}=M^{4k}\times \rp2\times S^1 $ does not admit a \tr immersion into 	
	$\bC ^N$ for $N=6k+3$.
	
\end{itemize}
\end{theorem}

Denote the total Chern class of a complex vector bundle $B$ over $M$ by
\[
c(B)=1+c_1(B)+\cdots +c_k(B)
\]
where $c_j(B) \in H^{2j}(M;\bZ)$ and 
$k=\min (\rank B, [\frac{\dim M}{2}])$.  We have the following well-known result (see e.g. \cite[Section 14]{MS}).
\begin{lemma}\label{agenerator}
Let $a$ denote the first Chern class of the hyperplane line bundle 
$\mathcal{O}(1)$ on $\cp2$.  Then
\[
c(\bC \otimes T\cp2) = 1-3a^2.
\]
\end{lemma}
We need to show that  in the first two cases of Theorem \ref{M} there is no bundle $Q$ of rank $2k-1$ and in the last two cases no bundle $Q$ of rank $2k$ such that $(\ctm)\oplus Q$ is trivial.  We shall show this for $M^{4k+1}$ and $M^{4k+3}$.  The other two cases, which are very similar to these, are done in \cite{HJL}.  So first we assume that there is some $Q$ with
\[
(\ctm ^{4k+1})\oplus Q \cong N\varepsilon
\]
for some $N$ and show that the rank of $Q$ is at least $2k$.

Let $a_1,\ldots ,a_k$ be the pull-backs of $a$ to $M$ under the corresponding projections to $\cp2$, so that $a_i^3=0$ for all $i$.  We have 
\[
c(\ctm ^{4k+1})\cdot c(Q)=1 .
\]
Thus $c(Q)= (1+3a_1^2)\cdots (1+3a_k^2)$.  Since $a_1^2\cdots a_k^2\neq 0$, this implies that the rank of $Q$ is at least $2k$.

Next we assume that there exists some $Q$ with 
\[
(\ctm ^{4k+3})\oplus Q=N\varepsilon
\]
for some $N$ and show that the rank of $Q$ is at least $2k+1$.  Let $a_1,\ldots a_k$ be as before and let $b_1$ be the pull-back of the generator in 
$H^2(\rp2;\bZ)$ given by the Chern class of the complexification of the tautological line bundle on $\rp2$.  We have
\[
c(\ctm ^{4k+3})\cdot c(Q)=1
\]
which now gives
\[
c(Q)= (1+3a_1^2)\cdots (1+3a_k^2)(1-b_1).
\]
This implies that the rank of $Q$ is at least $2k+1$.

\medskip
\subsection{Independent maps}
The same manifolds $M^{4k+r}~ (0 \leq r \leq 3)$ show that Theorem \ref{ind} is also optimal.
\begin{theorem}\ 
\begin{itemize}
	\item 
	$M^{4k} $ does not admit an \im  into $\bC ^N$ for $N=2k+1$.
	
	\item 
	$M^{4k+1}=M^{4k}\times S^1 $ does not admit an \im  into 
	$\bC ^N$ for $N=2k+2$.

	\item 
	$M^{4k+2}=M^{4k}\times \rp2 $ does not admit an \im  into 
	$\bC ^N$ for $N=2k+2$.
	
	\item 
	$M^{4k+3}=M^{4k}\times \rp2\times S^1 $ does not admit an \im into 	
	$\bC ^N$ for $N=2k+3$.
	
\end{itemize}
\end{theorem}
The proofs are similar to those of Theorem \ref{M} and can be found in \cite{HJL}.  For instance, to show that $M^{4k+1}$ does not admit an \im into $\bC ^N$ for $N=2k+2$ we start with
\[
\ctm ^{4k+1}\cong N\varepsilon \oplus B \]
for some $N$ which gives us 
\[
c(B)=c(\ctm ^{4k+1})= (1-3a_1)^2\cdots (1-3a_k^2).
\]
So the rank of $B$ is at least $2k$ and since $N+\text{rank}\, B=4k+1$, this leads to $N\leq 2k+1$.


\section {New results for four-manifolds}\label{new}

The fact that the same set of examples demonstrates the optimality of both Theorems \ref{tri} and \ref{ind} suggests that for some class of manifolds the two conditions
\[
(\ctm) \oplus Q \cong N\varep .
\]
and
\[
\ctm \cong N\varepsilon \oplus B.
\]
are related.  As a first step in exploring this, we present a result for 
four-dimensional manifolds.  

\begin{theorem}\label{iff}
Let $M$ be either an open or an orientable four-manifold.  Then $M$ has a \tr immersion into $\bC ^5$ \iff $M$ admits an \im  into $\bC ^3$.
\end{theorem}

This result is false (in both directions) for non-orientable $4$-manifolds:

\begin{theorem}\label{sharp}
$\mathbf{R} \mathrm{P}^4$ admits an \im into $\bC ^3$, but no \tr immersion into 
$\bC ^5$.  Moreover, the connected sum of $\mathbf{R} \mathrm{P}^4$ and 
$\rp2 \times \rp2$ admits a \tr immersion into $\bC ^5$, but no \im  into $\bC ^3$.  
\end{theorem}

\begin{proof}[Proof of Theorem \ref{iff}]
The hypothesis on $M$ implies that $H^4(M;\mathbf{Z})$ is either zero or is torsion-free. We will also use that $2c_1(\ctm) =0.$ 

Let $M$ have such a \tr immersion. So there exists a $Q$ of rank 1 with
\begin{equation}\label{Q}
(\ctm) \oplus Q \cong 5\varep
\end{equation}
and we want to find a $B$ (also of rank 1) such that
\begin{equation}\label{B}
\ctm \cong 3\varep \oplus B .
\end{equation}
From
\[
c((\ctm) \oplus Q)=1
\]
we derive
\[
c_2(\ctm) ={c_1(\ctm )}^2 .
\]
Thus 
\[
2c_2(\ctm) =0
\]
which implies
\[
c_2(\ctm) =0.
\]
Dimensional considerations imply that  $\ctm \cong 2\varep \oplus B'$, where $B'$ is a rank two complex vector bundle.  Note that $c_2(B')=c_2(\ctm)$ and so $c_2(B')=0$.  Since $c_2(B')$ coincides with the Euler characteristic of the underlying real oriented bundle, $B'$ admits a global nowhere zero section  
(see \cite[Theorem 12.5]{MS}).  Thus
\[
\ctm \cong 2\varep \oplus B' \cong 2\varep \oplus\varep\oplus B=3\varep \oplus B
\]
and so M admits an independent map into $\bC ^3$.

Now, conversely, we start with
\[
\ctm \cong 3\varep \oplus B
\]
for some $B$ of rank 1
and prove that there exists some $Q$ (also of rank 1) with
\[
(\ctm) \oplus Q \cong 5\varep .
\]
We see that
\[
c(\ctm) = c(B). 
\]
This yields
\[
c_1(\ctm)= c_1(B),
\]
so that 
\[
2 c_1(B) = 0
\]
which implies that 
\[
(c_1(B))^2 = 0
\]
by the hypothesis on $M$.
From Theorem \ref{tri} and Lemma \ref{BB}(a) we have
\[
(\ctm) \oplus Q'\cong 6 \varep
\]
for some $Q'$ of rank 2.  It remains to show that $Q' \cong Q \oplus \varep$, which will follow from $c_2(Q') = 0$; the latter equation holds since 
\[
c(Q') = c(B)^{-1} = (1 + c_1(B))^{-1} = 1 + c_1(B).
\] 
\end{proof}

\begin{remark}\normalfont
We further observe that when (\ref{Q}) and (\ref{B}) hold, it follows that $Q \cong B$ and that these line bundles have trivial square since their first Chern classes have order 2.  Moreover, from the Corollary below we see that when $M$ is orientable these line bundles are in fact trivial, so that $M$ admits a totally real immersion into 
$\bC ^4$.   
\end{remark}

We use the following lemma in the proof of Theorem \ref{sharp}.

\begin{lemma}

\begin{enumerate}
\item[(a)]
$M^4$ has a \tr immersion into $\bC ^5$ \iff the dual Pontryagin class $\overline{p}_1(M)$ is zero.
\item[(b)]
$M^4$ admits an \im into $\bC ^3$ \iff the Pontryagin class $p_1(M)$ is zero.
\end{enumerate}
\end{lemma}

\begin{proof}

Assume that $M^4$ has a \tr immersion into $\bC^5$, so there exists a line bundle 
$Q$ such that \eqref{Q} holds.  
The dual Pontryagin class is defined by taking any immersion of $M$ into some 
$\mathbf{R}^m$ and setting $\overline p_1(M)= - c_2(\mathbf{C}\otimes N)$, where $N$ is the normal bundle of the immersion.  Thus $TM\oplus N$ is trivial.  Since
\[
(\ctm) \oplus (\bC\otimes N)
\]
is also trivial, we see that $c_2(\bC\otimes N) = c_2(Q) = 0$.  Thus $\overline{p}_1(M)=0$. 

Conversely, assume $\overline{p}_1(M) = 0$.  As before, we have  
$(\ctm) \oplus Q'\cong 6 \varep$ with $Q'$ of rank 2.  We have 
$\overline{p}_1(M) = -c_2(Q')$, so $Q' \cong Q \oplus \varep$ and we then can conclude that $(\ctm) \oplus Q$ is trivial. 
  
On the other hand, the first Pontryagin class of $M$ is equal, up to sign, to 
$c_2(\ctm )$.  Thus if 
\[
\ctm =3\varep \oplus B,
\]
with $B$ of rank 1, we have $p_1(M)=0$.

Finally, suppose that $p_1(M) = 0$, i.e. $c_2(\ctm) = 0$. We can write
$\ctm \cong 2\varep \oplus B'$, where $\rank B' = 2$.  So $c_2(B') = 0$, which implies that $B' \cong \varep \oplus B$, yielding $\ctm \cong 3\varep \oplus B$, as required.  
\end{proof}

Theorem \ref{sharp} is now an immediate consequence of the following lemma.

\begin{lemma}
$p_1(\mathbf{R}\mathrm{P}^4) = 0$ and 
$\overline{p}_1(\mathbf{R} \mathrm{P}^4)\neq 0$, while the opposite is true for the connected sum of $\mathbf{R} \mathrm{P}^4$ and $\rp2 \times \rp2$.
\end{lemma}

\begin{proof}

Note first that for a closed connected non-orientable $4$-manifold $M$ the coefficient homomorphism
$H^4(M; \bZ) \to H^4(M; \bZ_2)$ induced by reduction mod 2 is an isomorphism, as follows from the long exact sequence induced by the short exact sequence 
$$
0 \to \bZ \to \bZ \to \bZ_2 \to 0
$$ 
of coefficient groups, making use of the fact that $H^4(M; \bZ) \cong \bZ_2$.  It is well known that for any real vector bundle over $M$ this coefficient homomorphism sends $p_1$ to $(w_2)^2$ (see e.g. \cite[Problem 15-A]{MS}).  So $p_1(M) = 0$ \iff ($(w_2(M))^2 = 0$; and 
$\overline{p}_1(M) = 0$ \iff ($(\overline{w}_2(M))^2 = 0$.

\medskip
For $\mathbf{R}\mathrm{P} ^4$ we have 
$w(\mathbf{R}\mathrm{P} ^4) = (1+x)^5 = 1 + x+ x^4$, where $x$ denotes the generator in 1-dimensional cohomology.  So $w_2(\mathbf{R}\mathrm{P} ^4) = 0$ and therefore also $p_1(\mathbf{R}\mathrm{P} ^4) = 0$.  In addition, 
$\overline{w}(\mathbf{R}\mathrm{P} ^4) = 1 + x +x^2 +x^3$, so 
$\overline{w}_2(\mathbf{R}\mathrm{P} ^4) = x^2$ and it follows that 
$\overline{p}_1(\mathbf{R}\mathrm{P} ^4) \neq 0.$ 

\medskip
Now let $M = \mathbf{R}\mathrm{P} ^4 \# (\rp2 \times \rp2)$.  As this manifold is cobordant to the disjoint union of its two ``summands", its Stiefel-Whitney and dual Stiefel-Whitney numbers are the sums of the corresponding characteristic numbers of its summands.  We have determined these characteristic numbers for 
$\mathbf{R} \mathrm{P}^4$, and need only add to them the corresponding characteristic numbers for $\rp2 \times \rp2$.  One easily computes that for both $(w_2)^2$ and 
$(\overline{w}_2)^2$ the characteristic numbers of $\rp2 \times \rp2$ are nonzero.  
It follows that $p_1(M) \neq 0$, while $\overline{p}_1(M) = 0$.   
\end{proof}

In the case of an orientable four-manifold, we can obtain the following more precise result.

\begin{corollary}
Let $M$ be an orientable $4$-manifold.  Then $p_1(M)$ is zero \iff $\overline{p}_1(M)$ is zero, and these conditions are equivalent to the existence of a \tr immersion of 
$M$ into $\bC^4$. When these conditions fail, $M$ admits no \tr immersion into 
$\bC^5$, nor an \im \ into $\bC^3$.   	
\end{corollary}

\begin{proof}
Let $M$ be an orientable 4-manifold.  Suppose that $M$ admits a \tr immersion into 
$\bC ^5$.  In this case, we know that $(\ctm) \oplus Q$ is trivial for a complex line bundle $Q$.  Hence  we have $c_1(\ctm) = - c_1(Q)$.

Since $M$ is orientable, the top exterior power of $TM$ is a trivial real line bundle, hence the top exterior power of $\ctm$ is a trivial complex line bundle, and so (e.g., by using the splitting principle) $c_1(\ctm) = 0.$  Hence $Q$ is a trivial line bundle.

It follows that $\ctm$ is stably trivial, and therfore is trivial for dimensional reasons.  In turn, this implies that $M$ admits a \tr immersion into $\bC^4.$ 

\medskip 
Similarly, if an orientable $4$-manifold $M$ admits an independent map into $\bC ^3$,   
then in fact $M$ admits a \tr immersion into $\bC^4$.
\end{proof}

\begin{remark}\normalfont
We provide a summary of the results in this section for a four-dimensional manifold 
$M$, in terms of the following list of conditions that the manifold may satisfy:

\begin{enumerate}

	\item $M$ admits a \tr immersion into $\bC^5$.
	
	\item $M$ admits a \tr immersion into $\bC^4$.
	
	\item $M$ admits an \im into $\bC^3$.
	
	\item $M$ admits an \im into $\bC^4$.
	
	\item $\ctm$ is trivial.
	
	\item The first dual Pontryagin class of $M$ vanishes.
	
	\item The first Pontryagin class of $M$ vanishes. 
		
\end{enumerate}

\noindent
Then:

\begin{enumerate}

	\item[(a)]  Conditions (2), (4), and (5) are equivalent for all $4$-manifolds,
	and plainly imply the remaining conditions.   
	
	\item[(b)]  Conditions (1) and (6) are equivalent for all $4$-manifolds.  The 
	same holds for Conditions (3) and (7).
	
	\item[(c)]  Conditions (1), (3), (6), and (7) are all satisfied if $M$ is open. 	
	\item[(d)]  All seven conditions are equivalent if $M$ is orientable.
	
	\item[(e)]  By Theorem 4.2, conditions (1) and (3) are not equivalent for compact 
	non-orientable manifolds; indeed, neither implies the other.
	
	\item[(f)] The conditions (1), (3), (6), and (7) are satisfied by the 
	non-orientable manifolds $\rp2 \times \bR^2$ and $\rp2 \times S^2$, but these 
	manifolds do not satisfy the conditions (2), (4), and (5), since in both case the 
	first Chern class of the complexified tangent bundle is nonzero. 
		
\end{enumerate}
It seems unlikely that such complete results can be obtained for manifolds of larger dimension.	
 
\end{remark}


\section{A geometric approach to Theorem \ref{iff}}

Here is an alternative proof that the equation (\ref{B}),
\[
\ctm =3\varep \oplus B,
\]
implies, for orientable $M^4$, that there exists some $Q$ satisfying equation 
(\ref{Q}), 
\[
(\ctm) \oplus Q= 5\varepsilon .
\]
The zero set of a generic section $\sigma_1$ of the line bundle $B$ is a (possibly not connected) 2-dimensional orientable submanifold $Y\subset M$.  In the usual way we have $2c_1(B)=0$ and so this is also true for the restriction of $B$ to $Y$.  But since $Y$ is  orientable  we may conclude that also $c_1(B|_Y)=0$, which implies that $B|_Y$ is trivial.  Let $\sigma _2$ be a nonzero section of $B|_Y$ and extend $\sigma _2$ smoothly to a section over $M$.  These two sections provide a fiber-surjective map $M\times \bC ^2\rightarrow B$.  In light of \eqref{B}, we then have a fiber-surjective map
\[
M\times \bC ^5\rightarrow \ctm 
\]  
and this leads to \eqref{Q}.


\end{document}